\documentclass[oneside,english]{amsart}
\usepackage[T1]{fontenc}
\usepackage[latin9]{inputenc}
\usepackage{amsthm}
\usepackage{amssymb}

\makeatletter
\numberwithin{equation}{section}
\numberwithin{figure}{section}
\theoremstyle{plain}
\newtheorem{thm}{\protect\theoremname}
  \theoremstyle{plain}
  \newtheorem{fact}[thm]{\protect\factname}
  \theoremstyle{remark}
  \newtheorem{rem}[thm]{\protect\remarkname}
  \theoremstyle{plain}
  \newtheorem{lem}[thm]{\protect\lemmaname}
  \theoremstyle{plain}
  \newtheorem{cor}[thm]{\protect\corollaryname}
  \theoremstyle{plain}
  \newtheorem{prop}[thm]{\protect\propositionname}
  \theoremstyle{definition}
  \newtheorem{defn}[thm]{\protect\definitionname}
  \theoremstyle{remark}
  \newtheorem*{claim*}{\protect\claimname}
  \theoremstyle{definition}
  \newtheorem{example}[thm]{\protect\examplename}
  \theoremstyle{definition}
  \newtheorem{problem}[thm]{\protect\problemname}
  \theoremstyle{remark}
  \newtheorem{claim}[thm]{\protect\claimname}

\usepackage[mathcal,mathbf]{euler}
\thanks{The first author was supported by the Marie Curie Initial Training Network in Mathematical Logic - MALOA - From MAthematical LOgic to Applications, PITN-GA-2009-238381}

\makeatother

\usepackage{babel}
  \providecommand{\claimname}{Claim}
  \providecommand{\corollaryname}{Corollary}
  \providecommand{\definitionname}{Definition}
  \providecommand{\examplename}{Example}
  \providecommand{\factname}{Fact}
  \providecommand{\lemmaname}{Lemma}
  \providecommand{\problemname}{Problem}
  \providecommand{\propositionname}{Proposition}
  \providecommand{\remarkname}{Remark}
\providecommand{\theoremname}{Theorem}

\begin{document}
\global\long\def\NIP{\operatorname{NIP}}

\global\long\def\IP{\operatorname{IP}}

\global\long\def\P{\operatorname{\mathbf{P}}}

\global\long\def\ind{\operatorname{ind}}

\global\long\def\eq{\operatorname{eq}}

\global\long\def\fs{\operatorname{fs}}

\global\long\def\acl{\operatorname{acl}}

\global\long\def\dcl{\operatorname{dcl}}

\global\long\def\tp{\operatorname{\mbox{tp}}}

\global\long\def\bdd{\operatorname{bdd}}

\global\long\def\Th{\operatorname{\mbox{Th}}}

\global\long\def\nfcp{\operatorname{\mbox{nfcp}}}

\global\long\def\UDTFS{\operatorname{UDTFS}}

\global\long\def\NFCP{\operatorname{nfcp}}

\global\long\def\dnfcp{\operatorname{dnfcp}}

\global\long\def\M{\operatorname{\mathbb{M}}}

\global\long\def\VC{\operatorname{VC}}

\global\long\def\VCdim{\operatorname{VC}}

\global\long\def\VCcodim{\operatorname{coVC}}

\title{Externally definable sets and dependent pairs II}

\author{Artem Chernikov
\and
Pierre Simon}
\begin{abstract}
We continue investigating the structure of externally definable sets
in $\NIP$ theories and preservation of $\NIP$ after expanding by
new predicates. Most importantly: types over finite sets are uniformly
definable; over a model, a family of non-forking instances of a formula
(with parameters ranging over a type-definable set) can be covered
with finitely many invariant types; we give some criteria for the
boundedness of an expansion by a new predicate in a distal theory;
naming an arbitrary small indiscernible sequence preserves $\NIP$,
while naming a large one doesn't; there are models of $\NIP$ theories
over which all 1-types are definable, but not all n-types.
\end{abstract}
\maketitle

\section*{Introduction}

A characteristic property of stable theories is the definability of
types. Equivalently, every externally definable set is internally
definable. In unstable theories this is no longer true. However, as
was observed early on by Shelah (e.g. \cite{Sh783}), the class of
externally definable sets in $\NIP$ theories satisfies some nice
properties resembling those in the stable case (e.g. it is closed
under projection). In this paper we continue the investigation of
externally definable sets in $\NIP$ theories started in \cite{ExtDefI}.

As it was established there, every externally definable set $X=\phi(x,b)\cap A$
has an \emph{honest definition}, which can be seen as the existence
of a uniform family of internally definable subsets approximating
$X$. Formally, there is $\theta(x,z)$ such that for any \emph{finite}
$A_{0}\subseteq X$ there is some $c\in A$ satisfying $A_{0}\subseteq\theta(A,c)\subseteq A$.
The first section of this paper is devoted to establishing the existence
of \emph{uniform} honest definitions. By uniform we mean that $\theta(x,z)$
can be chosen depending just on $\phi(x,y)$ and not on $A$ or $b$.
We achieve this assuming that the whole theory is $\NIP$, combining
careful use of compactness with a strong combinatorial result of Alon-Kleitman
\cite{MR1185788} and Matousek \cite{MR2060639}: the $\left(p,k\right)$-theorem.
As a consequence we conclude that in an $\NIP$ theory types over
finite sets are uniformly definable ($\UDTFS$). This confirms a conjecture
of Laskowski.

In the next section we consider an implication of the $(p,k)$-theorem
for forking in $\NIP$ theories. Combined with the results on forking
and dividing in $\NIP$ theories from \cite{CheKap}, we deduce the
following: working over a model $M$, let $\left\{ \phi(x,a)\,:\, a\models q(y)\right\} $
be a family of non-forking instances of $\phi(x,y)$, where the parameter
$a$ ranges over the set of solutions of a partial type $q$. Then
there are finitely many global $M$-invariant types such that each
$\phi(x,a)$ from the family belongs to one of them.

In Section 3 we return to the question of naming subsets with a new
predicate. In \cite{ExtDefI} we gave a general condition for the
expansion to be $\NIP$: it is enough that the theory of the pair
is \emph{bounded}, i.e. eliminates quantifiers down to the predicate,
and the induced structure on the predicate is NIP. Here, we try to
complement the picture by providing a general sufficient condition
for the boundedness of the pair. In the stable case the situation
is quite neatly resolved using the notion of $\nfcp$. However $\nfcp$
implies stability, so one has to come up with some generalization
of it that is useful in unstable $\NIP$ theories. Towards this purpose
we introduce \emph{dnfcp}, i.e. no finite cover property for definable
sets of parameters, and its relative version with respect to a set.
We also introduce dnfcp' -- a weakening of dnfcp with separated variables.
Using it, we succeed in the distal, stably embedded, case: if one
names a subset of $M$ which is small, uniformly stably embedded and
the induced structure satisfies dnfcp', then the pair is bounded.

In section 4 we look at the special case of naming an indiscernible
sequence. On the one hand, we complement the result in \cite{ExtDefI}
by showing that naming a small indiscernible sequence of \emph{arbitrary}
order type is bounded and preserves $\NIP$. On the other hand, naming
a large indiscernible sequence does not.

In the last section we consider models over which all types are definable.
While in general even $o$-minimal theories may not have such models,
many interesting $\NIP$ theories do ($RCF$, $ACVF$, $\Th(\mathbb{Q}_{p})$,
Presburger arithmetic...). In practice, it is often much easier to
check definability of $1$ types, as opposed to $n$-types, so it
is natural to ask whether one implies the other. Unfortunately, this
is not true -- we give an $\NIP$ counter-example. Can anything be
said on the positive side? Pillay \cite{MR2830421} had established:
let $M$ be NIP, $A\subseteq M$ be definable with rosy induced structure.
Then if it is 1-stably embedded, it is stably embedded. We observe
that Pillay's results holds when the definable set $A$ is replaced
with a model, assuming that it is \emph{uniformly} 1-stably embedded.
This provides a generalization of the classical theorem of Marker
and Steinhorn about definability of types over models in $o$-minimal
theories. We also remark that in $\NIP$ theories, there are arbitrary
large models with ``few'' types over them (i.e. such that $\left|S(M)\right|\leq\left|M\right|^{\left|T\right|}$).

\section*{Preliminaries}

\subsection{VC dimension, co-dimension and density}

Let $\mathcal{F}$ be a family of subsets of some set $X$. Given
$A\subseteq X$, we say that it is \emph{shattered} by $\mathcal{F}$
if for every $A'\subseteq A$ there is some $S\in\mathcal{F}$ such
that $A\cap S=A'$.

A family $\mathcal{F}$ is said to have finite \emph{$\VC$-dimension}
if there is some $n\in\omega$ such that no subset of $X$ of size
$n$ can be shattered by $\mathcal{F}$. In this case we let $\VCdim(\mathcal{F})$
be the largest integer $n$ such that some subset of $X$ of size
$n$ is shattered by it.

The $\VC$ co-dimension of $\mathcal{F}$ is the largest integer $n$
for which there are $S_{1},...,S_{n}\in\mathcal{F}$ such that for
any $u\subseteq n$ there is $b_{u}\in X$ satisfying $b_{u}\in S_{i}\Leftrightarrow i\in u$.
It is well known that $\VCcodim(\mathcal{F})<2^{\VCdim(\mathcal{F})+1}$.

\subsection{NIP and alternation}

We are working in a monster model $\M$ of a complete first-order
theory $T$.

Recall that a formula $\phi(x,y)$ is $\NIP$ if there are no $\left(a_{t}\right)_{t\in\omega}$
and $\left(b_{s}\right)_{s\subseteq\omega}$ such that $\phi(a_{t},b_{s})\Leftrightarrow t\in s$.
Equivalently, for any indiscernible sequence $\left(a_{t}\right)_{t\in I}$
and $b$, there can be only finitely many $t_{0}<...<t_{n}\in I$
such that $\phi(a_{t_{i}},b)\Leftrightarrow i$ is even. The following
is a very important refinement of this statement, see e.g. \cite[Theorem 14]{Adl}. 

Let $\left(a_{t}\right)_{t\in I}$ be an indiscernible sequence and
let $E$ be a convex equivalence relation on $I$. If $\bar{t}=\left(t_{i}\right)_{i<\kappa}$
and $\bar{s}=\left(s_{i}\right)_{i<\kappa}$ are tuples of elements
from $I$, we will write $\bar{t}\sim_{E}\bar{s}$ if $\bar{t}$ and
$\bar{s}$ have the same quantifier-free order type and $t_{i}Es_{i}$
for all $i<\kappa$.
\begin{fact}
\label{fac: BB} Let $\left(a_{t}\right)_{t\in I}$ be an indiscernible
sequence and let $b$ be any finite tuple. Let $\phi(x_{0},...,x_{n};y)$
be $\NIP$. Then there is a convex equivalence relation $E$ on $I$
with finitely many classes such that for any $\left(s_{i}\right)_{i\leq n}\sim_{E}\left(t_{i}\right)_{i\leq n}$
from $I$ we have $\phi(a_{s_{0}},...,a_{s_{n}};b)\leftrightarrow\phi(a_{t_{0}},...,a_{t_{n}};b)$.\end{fact}
\begin{rem}
\label{rem: types over complete sequence are definable}In particular,
if $I$ is a complete linear order and $\phi(x_{0},...,x_{n};y)$
is $\NIP$, then all $\phi$-types over $I$ are definable, possibly
after adding finitely many elements extending $I$ on both sides.
Why? If $I$ is totally indiscernible, then all $\phi$-types over
it are in fact definable using just equality. If it is not, then there
is some formula giving the order on the sequence, and by Fact \ref{fac: BB},
$\phi$-types over $I$ are definable using this order (see \cite[Section 3.1]{ExtDefI}).
\end{rem}
In a natural way we define the $\VC$ dimension of a formula in a
model $M$ as $\VCdim(\phi(x,y))=\VCdim\left\{ \phi(M,a)\,:\, a\in M^{n}\right\} $.
Notice that this value does not depend on the model, so we'll talk
about $\VCdim$ dimension of $\phi$ in $T$. Similarly we define
$\VC$ co-dimension.

It was observed early on by Laskowski that $\phi(x,y)$ is $\NIP$
if and only if it has finite $\VC$ dimension, if and only if it has
finite $\VC$ co-dimension \cite{MR1171563}. We also recall an early
result of Shelah about counting types over finite sets.
\begin{fact}
\label{fac: [Shelah/Sauer]} {[}Shelah/Sauer{]} The following are
equivalent:
\begin{enumerate}
\item $\phi(x,y)$ is $\NIP$. 
\item There are $k,d\in\omega$ such that for all finite $A$, $\left|S_{\phi}(A)\right|\leq d\cdot\left|A\right|^{k}$.
\end{enumerate}
\end{fact}
Then one defines the \emph{$\VC$ density} of $\phi$ to be the infimum
of all reals $r$ such that for some $d$, $\left|S_{\phi}(A)\right|\leq d\cdot\left|A\right|^{r}$
for all finite $A$.

\subsection{Invariant types}

Let $p(x)$ be a global type over a monster model $\M$, invariant
over some small submodel $M$. Then one naturally defines $p^{(\omega)}(x)\in S_{\omega}(\M)$,
the type of a Morley sequence in it (see \cite[Section 2]{MR2800483}
for details).
\begin{fact}
\label{fac: MS determines invariant type} Let $T$ be $\NIP$. Assume
that $p(x),q(x)$ are global types invariant over a small model $M$.
If $p^{(\omega)}|_{M}=q^{(\omega)}|_{M}$, then $p=q$.
\end{fact}
We will use the following lemma, see \cite[Lemma 2.18]{Distal} for
a proof.
\begin{lem}
\label{lem: weak local character} Assume that $T$ is NIP. Let $a$
be given and $q(x)\in S(A')$ be invariant over $C\subset A'$. Then
there is $D$ of size $\leq|C|+|x|+|a|+|T|$ such that $C\subseteq D\subseteq A'$
and for any $b,b'\in A'$ realizing $q(x)|D$, $\tp(ab/D)=\tp(ab'/D)$.
\end{lem}

\subsection{(p,k)-theorem}

We will need the following theorem from \cite{MR2060639}.
\begin{fact}
\label{Fac: (p,k)-theorem}{[}$(p,k)$-theorem{]} Let $\mathcal{F}$
be a family of subsets of some set $X$. Assume that the VC co-dimension
of $\mathcal{F}$ is bounded by $k$. Then for every $p\geq k$, there
is an integer $N$ such that: for every finite subfamily $\mathcal{G}\subset\mathcal{F}$,
if $\mathcal{G}$ has the $(p,k)$-property meaning that among any
$p$ subsets of $\mathcal{G}$ some $k$ intersect, then there is
an $N$-point set intersecting all members of $\mathcal{G}$. \end{fact}
\begin{rem}
\label{Rem: In (p,k) theorem N does not depend on F} Although the
theorem is stated this way in \cite{MR2060639}, $N$ depends only
on $p$ and $k$ and not on the family $\mathcal{F}$. To see this,
assume that for every $N$, we had a family $\mathcal{F}_{N}$ on
some set $X_{N}$ of VC co-dimension bounded by $k$ and for which
the $(p,k)$ theorem fails for this $N$. Then consider $X$ to be
the disjoint union of the sets $X_{N}$ and $\mathcal{F}$ the union
of the families $\mathcal{F}_{N}$. Then clearly $\mathcal{F}$ has
VC co-dimension bounded by $k$ and the theorem fails for it. Also,
it follows from the proof.
\end{rem}

\subsection{Expansions and stable embeddedness}

Let $A$ be a subset of $M\models T$ and let $L_{\P}=L\cup\{\P(x)\}$,
where $\P(x)$ is a new unary predicate. We define the structure $(M,A)$
as the expansion of $M$ to an $L_{\P}$-structure where $\P(M)=A$.
Recall that $\Th(M,A)$ is \emph{$\P$-bounded} if every $L_{\P}$
formula is equivalent to one of the form
\[
Q_{1}y_{1}\in\P...Q_{n}y_{n}\in\P\phi(x,\bar{y}),
\]

where $Q_{i}\in\left\{ \exists,\forall\right\} $ and $\phi$ is an
$L$-formula. We may just say bounded when it creates no confusion.

Given $A\subseteq M\models T$ and a set of formulas $F$, possibly
with parameters, we let $A_{\ind(F)}$ be the structure in the language
$L(T)\cup\left\{ D_{\phi(x)}(x)\,:\,\phi(x)\in F\right\} $ with $D_{\phi}(x)$
interpreted as the set $\phi(A)$. When $F=L$, we may omit it. Given
$A\subseteq M$ and a tuple $b\in M$, let $A_{[b]}$ be shorthand
for $A_{\ind(F)}$ with $F=\left\{ \phi(x,b)\,:\,\phi\in L\right\} $. 

A set $A\subset M$ is called \emph{small} if for every finite $b\in M$,
every finitary type over $Ab$ is realized in $M$. Finally, a set
$A\subset M$ is \emph{stably embedded} if for every $\phi(x,y)$
and $c\in M$ there is $\psi(x,z)$ and $b\in A$ such that $\phi(A,c)=\psi(A,b)$.
We say that it is \emph{uniformly stably embedded} if $\psi$ can
be chosen depending just on $\phi$, and not on $c$. A definable
set is stably embedded if and only if it is uniformly stably embedded,
by compactness.

\section{Uniform honest definitions}

\subsection{Uniform honest definitions}

~

We recall the following result about existence of honest definitions
for externally definable sets in $\NIP$ theories established in \cite{ExtDefI}. 
\begin{fact}
\label{fac: honest definitions exist} {[}Honest definition{]} Let
$T$ be $\NIP$ and let $M$ be a model of $T$ and $A\subseteq M$
any subset. Let $\phi(x,a)$ have parameters in $M$. Then there is
an elementary extension $(M',A')$ of the pair $(M,A)$ and a formula
$\theta(x,b)\in L(A')$ such that $\phi(A,a)=\theta(A,b)$ and $\theta(A',b)\subseteq\phi(A',a)$. 
\end{fact}
It can be reformulated as existence of a uniform family of internally
definable subsets approximating our externally definable set.
\begin{cor}
\label{cor: Approximate honest definitions} Let $M$, $A$ and $\phi(x,a)$
be as above. Then there is $\theta(x,t)$ such that for any finite
subset $A_{0}\subseteq\phi(A,a)$, there is $b\in A$ such that $A_{0}\subseteq\theta(A,b)\subseteq\phi(A,a)$. \end{cor}
\begin{proof}
Immediately follows from Fact \ref{fac: honest definitions exist}
because the extension $(M,A)\prec(M',A')$ is elementary and the condition
on $b$ can be stated as a single formula in the theory of the pair.
Note that conversely this implies Fact \ref{fac: honest definitions exist}
by compactness.
\end{proof}
It is natural to ask whether $\theta$ can be chosen in a uniform
way depending just on $\phi$, and not on $A$ and $a$ (Question
1.4 from \cite{ExtDefI}). The aim of this section is to answer this
question positively.

First, compactness gives a weak uniformity statement.
\begin{prop}
\label{prop: Weak uniformity for honest definitions}Fix a formula
$\phi(x,y)$. For every formula $\theta(x,t)$ (in the same variable
$x$, but $t$ may vary), fix an integer $n_{\theta}$. Then there
are finitely many formulas $\theta_{1}(x,t_{1}),...,\theta_{k}(x,t_{k})$
such that the following holds:

For every $M\models T$ and $A\subset M$, for every $a\in M$ there
is $i\leq k$ such that for every subset $A_{0}\subseteq\phi(A,a)$
of size at most $n_{\theta_{i}}$, there is $b\in A$ satisfying $A_{0}\subseteq\theta_{i}(A,b)\subseteq\phi(A,a)$.\end{prop}
\begin{proof}
Consider the theory $T'$ in the language $L'=L\cup\{P(x),c\}$ saying
that if $(M,A)\models T'$ (where $A=P(M)$), then $M\models T$ and
for every $\theta\in L$, there is a subset $A_{0}$ of $\phi(A,c)$
of size at most $n_{\theta}$ for which there does not exist a $b\in A$
satisfying $A_{0}\subseteq\theta(A,b)\subseteq\phi(A,a)$. By Corollary
\ref{cor: Approximate honest definitions}, $T'$ is inconsistent.
By compactness, we find a finite set of formulas as required.
\end{proof}
Combining this with the $(p,k)$-theorem we get the full result.
\begin{thm}
\label{thm: Uniform honest definitions} Let $T$ be $\NIP$ and $\phi(x,y)$
given. Then there is a formula $\chi(x,t)$ such that for every set
$A$ of size $\geq2$, tuple $a$ and finite subset $A_{0}\subseteq A$,
there is $b\in A$ satisfying:
\begin{enumerate}
\item $\phi(A_{0},a)=\chi(A_{0},b)$,
\item $\chi(A,b)\subseteq\phi(A,a)$.
\end{enumerate}
\end{thm}
\begin{proof}
By the usual coding tricks, using $|A|\geq2,$ it is enough to find
a finite set of formulas $\{\chi_{i}\}_{i<n}$ such that for every
finite set, one of them works.

For every formula $\theta(x,t)$, let $n_{\theta}$ be its VC dimension.
Proposition \ref{prop: Weak uniformity for honest definitions} gives
us a finite set $\{\theta_{1},...,\theta_{k}\}$ of formulas. Using
the previous remark, we may assume $k=1$ and write $\theta(x,t)=\theta_{1}(x,t)$.
Let $N$ be given by Fact \ref{Fac: (p,k)-theorem} taking $p=k=n_{\theta}$
(using Remark \ref{Rem: In (p,k) theorem N does not depend on F}). 

Let $A_{0}\subseteq A\subseteq M\models T$ and $a\in M$ be given,
$A_{0}$ is finite. Set $B\subseteq A^{\left|t\right|}$ be the set
of tuples $b\in A^{\left|t\right|}$ such that $\theta(A,b)\subseteq\phi(A,a)$.
Consider the family $\mathcal{F}=\{\theta(d,B)\,:\, d\in\phi(A_{0},a)\}$
of subsets of $B$. This is a finite family, and by hypothesis the
intersection of any $k$ members of it is non-empty. Therefore the
$(p,k)$-theorem applies and gives us $N$ tuples $b_{1},....,b_{N}\in B$
such that $\{b_{1},...,b_{N}\}$ intersects any set in $\mathcal{F}$.
Unwinding, we see that $\phi(A_{0},a)=\bigvee_{i\leq N}\theta(A_{0},b_{i})$
and $\bigvee_{i\leq N}\theta(A,b_{i})\subseteq\phi(A,a)$. So taking
$\chi(x,t_{1}...t_{N})=\bigvee_{i\leq N}\theta(x,t_{i})$ works.
\end{proof}

\subsection{UDTFS}

~

Recall the following classical fact characterizing stability of a
formula.
\begin{fact}
\label{Fac: Characterization of stability}The following are equivalent:
\begin{enumerate}
\item $\phi(x,y)$ is stable.
\item There is $\theta(x,z)$ such that for any $A$ and $a$, there is
$b\in A$ satisfying $\phi(A,a)=\theta(A,b)$.
\item There are $m,n\in\omega$ such that $\left|S_{\phi}(A)\right|\leq m\cdot|A|^{n}$
for any set $A$.
\end{enumerate}
\end{fact}

\begin{defn}
We say that $\phi(x,y)$ has $\UDTFS$ (Uniform Definability of Types
over Finite Sets) if there is $\theta(x,z)$ such that for every finite
$A$ and $a$ there is $b\in A$ such that $\phi(A,a)=\theta(A,b)$.
We say that $T$ satisfies $\UDTFS$ if every formula does.\end{defn}
\begin{rem}
If $\phi(x,y)$ has $\UDTFS$, then it is $\NIP$ (by Fact \ref{fac: [Shelah/Sauer]}).
\end{rem}
Comparing Fact \ref{Fac: Characterization of stability} and Fact
\ref{fac: [Shelah/Sauer]} naturally leads to the following conjecture
of Laskowski: assume that $\phi(x,y)$ is $\NIP$, then it satisfies
$\UDTFS$. It was proved for weakly $o$-minimal theories in \cite{MR2610477}
and for $dp$-minimal theories in \cite{Guingona_UDTFS}. An immediate
corollary of Theorem \ref{thm: Uniform honest definitions} is that
if the whole $T$ is $\NIP$, then every formula satisfies $\UDTFS$.
\begin{thm}
Let $T$ be $\NIP$. Then it satisfies $\UDTFS$.\end{thm}
\begin{proof}
Follows from Theorem \ref{thm: Uniform honest definitions} taking
$A_{0}=A$.\end{proof}
\begin{rem}
This does not fully answer the original question as our argument is
using more than just the dependence of $\phi(x,y)$ to conclude $\UDTFS$
for $\phi(x,y)$. Looking more closely at the proof of Fact \ref{fac: honest definitions exist},
we can say exactly how much NIP is needed. Depending on the VC dimension
of $\phi$, there is a finite set $\Delta_{\phi}$ of formulas for
which we have to require NIP consisting of formulas of the form $\psi(x_{1},...,x_{k})=\exists y\bigwedge_{i}\phi(x_{i},y)^{\epsilon(i)}$,
where $k$ is at most $\VCdim(\phi)+1$.
\end{rem}
$\UDTFS$ implies that in the statement of the $(p,k)$-theorem for
sets inside an $\NIP$ theory consistent pieces are uniformly definable.
\begin{cor}
Let $T$ be $\NIP$. For any $\phi(x,y)$ there is $\psi(y,z)$ and
$k\leq N<\omega$ such that: for every finite $A$, if $\left\{ \phi(x,a)\,:\, a\in A\right\} $
is $k$-consistent, then there are $c_{0},...,c_{N-1}\in A$ such
that $A=\bigcup_{i<N}\psi(A,c_{i})$ and $\left\{ \phi(x,a)\,:\, a\in\psi(A,c_{i})\right\} $
is consistent for every $i<N$.
\end{cor}

\subsection{Strong honest definitions and distal theories}
\begin{defn}
A theory $T$ is called \emph{distal} if it satisfies the following
property: Let $I+b+J$ be an indiscernible sequence, with $I$ and
$J$ infinite. For arbitrary $A$, if $I+J$ is indiscernible over
$A$, then $I+b+J$ is indiscernible over $A$. 
\end{defn}
The class of distal theories was introduced in \cite{Distal}, in
order to capture the class of dependent theories which do not contain
any ``stable part''. Examples of distal theories include ordered
$dp$-minimal theories and $\mathbb{Q}_{p}$.

We will say that $p(x),q(y)\in S(A)$ are \emph{orthogonal} if $p(x)\cup q(y)$
determines a complete type over $A$.
\begin{prop}
\label{prop: isolated extension}{[}Strong honest definition{]} Let
$T$ be distal, $A\subset M$ and $a\in M$ arbitrary. Let $\left(M',A'\right)\succ\left(M,A\right)$
be $|M|^{+}$-saturated. Then for any $\phi(x,y)$ there are $\theta(x,z)$
and $b\in A'$ such that $\models\theta(a,b)$ and $\theta(x,b)\vdash\tp_{\phi}(a/A)$.\end{prop}
\begin{proof}
Let $\left(M',A'\right)\succ\left(M,A\right)$ be $\kappa=\left|M\right|^{+}$-saturated,
we show that $p=\tp_{L}(a/A')$ is orthogonal to any $L$-type $q\in S(A')$
finitely satisfiable in a subset of size $<\kappa$. So take such
a $q$, finitely satisfiable in $C\subset A'$. By Lemma \ref{lem: weak local character},
there is some $D$ of size $<\kappa$, $C\subseteq D\subset A'$,
such that for any two realizations $I,I'\subset A'$ of $q^{\left(\omega\right)}|D$,
we have $\tp_{L}(aI/C)=\tp_{L}(aI'/C)$. Take some $I\models q^{\left(\omega\right)}|D$
in $A'$ (exists by saturation of $\left(M',A'\right)$ and finite
satisfiability) and $J\models q^{\left(\omega\right)}|\M$.
\begin{claim*}
$I+J$ is indiscernible over $aC$.\end{claim*}
\begin{proof}
As $q^{\left(\omega\right)}|\M$ is finitely satisfiable in $C$,
by compactness and saturation of $\left(M',A'\right)$ there is $J'\models q^{\left(\omega\right)}|aDI$
in $A'$. 

If $I+J$ is not $aC$-indiscernible, then $I'+J'$ is not $aC$-indiscernible
for some finite $I'\subset I$. As both $I'+J'$ and $J'$ realize
$q^{\left(\omega\right)}|D$ in $A'$, it follows that $J'$ is not
indiscernible over $aC$ -- a contradiction.
\end{proof}
Now, if $b\in\M$ is any realization of $q$, then $I+b+J$ is $C$-indiscernible.
By the claim and distality, $I+b$ is $aC$-indiscernible. It follows
that $\tp(b/Ca)$ is determined by $\tp(a/A')$. As we can always
take a bigger $C$, $\tp(b/A'a)$ is determined, so $p$ is orthogonal
to $q$ as required.

Consider the set $S^{\fs}(A',A)$ of $L$-types over $A'$ finitely
satisfiable in $A$. It is a closed subset of $S_{L}(A')$. By compactness,
there is $\theta(x,b)\in p(x)$ such that for any $a'\models\theta(x,b)$
and any $c\models q(y)\in S^{\fs}(A',A)$, $\models\phi(a,c)\leftrightarrow\phi(a',c)$.
This applies, in particular, to every $c\in A$ and thus $\theta(x,b)\vdash tp_{\phi}(a/A)$.\end{proof}
\begin{rem}
In fact, the argument is only using that every indiscernible sequence
in $A'$ is distal.\end{rem}
\begin{thm}
\label{thm: distal =00003D uniform strong honest definitions} The
following are equivalent:
\begin{enumerate}
\item $T$ is distal.
\item For any $\phi(x,y)$ there is $\theta(x,z)$ such that: for any $A$,
$a$ and a finite $C\subseteq A$, there is $b\in A$ such that $\models\theta(a,b)$
and $\theta(x,b)\vdash\tp_{\phi}(a/C)$
\end{enumerate}
\end{thm}
\begin{proof}
(1)$\Rightarrow$(2): It follows from Proposition \ref{prop: isolated extension}
that we have: For any finite $C\subset A$, there is $b\in A$ such
that $\models\theta(a,b)$ and $\theta(x,b)\vdash\tp_{\phi}(a/C)$.
Similarly to the proof of Theorem \ref{thm: Uniform honest definitions},
we can choose $\theta$ depending just on $\phi$. 

(2)$\Rightarrow$(1): Let $I+d+J$ be an indiscernible sequence, with
$I$ and $J$ infinite. Assume that $I+J$ is indiscernible over $A$,
and we show that $I+d+J$ is indiscernible over $A$. 

Let $a$ be a finite tuple from $A$ and $\phi(x,y_{0}...y_{n}...y_{2n})\in L$,
and let $\theta(x,z)$ be as given for $\phi$ by (2). Without loss
of generality $\models\phi(a,b_{0}...b_{n}...b_{2n})$ holds for all
$b_{0}<...<b_{2n}\in I+J$. Let $I_{0}\subset I$ be finite. Then
for some $b\subset I_{0}$, $\models\theta(a,b)$ and $\theta(a,b)\vdash\tp_{\phi}(a/I_{0})$.
If we take $I_{0}$ to be large enough compared to $\left|z\right|$,
then there will be some $b_{0}<...<b_{n}<...<b_{2n}$ such that $\left\{ b_{i}\right\} _{i\leq2n}\cap b=\emptyset$.
As we have $\models\forall x\,\theta(x,b)\rightarrow\phi(x,b_{0}...b_{n}...b_{2n})$,
by indiscernibility of $I+d+J$ for any $\left\{ b'_{i}\right\} _{i\leq2n,i\neq n}$
in $I+J$ there is a corresponding $b'$ in $I+J$ such that $\models\forall x\,\theta(x,b')\rightarrow\phi(x,b'_{0}...d...b'_{2n})$.
As $\models\theta(a,b')$ holds by indiscernibility of $I+J$ over
$a$, it follows that $\models\phi(a,b_{0}...d...b_{2n})$ holds --
as wanted.\end{proof}
\begin{rem}
It follows from this theorem that types over finite sets in distal
theories admit uniform definitions of a special ``coherent'' form
as considered in \cite[Section 7.1]{VCdensity_50people}.
\end{rem}

\section{(p,k)-theorem and forking}

We recall some properties of dividing and forking in $\NIP$ theories.
\begin{fact}
\label{Properties of forking in NIP} Let $T$ be $\NIP$.
\begin{enumerate}
\item If $M\models T$, then $\phi(x,a)$ divides over $M$ $\Leftrightarrow$
it forks over $M$ $\Leftrightarrow$ the set $\left\{ \phi(x,a')\,:\, a\equiv_{M}a'\in\M\right\} $
is inconsistent.
\item For any $\phi(x,y)$, the set $\left\{ a:\phi(x,a)\mbox{ forks over }M\right\} $
is type-definable over $M$.
\item If $\left(a_{i}\right)_{i<\omega}$ is indiscernible over $M$ and
$\phi(x,a_{0})$ does not fork over $M$, then $\{\phi(x,a_{i})\}_{i<\omega}$
does not fork over $M$.
\item $\phi(x,a)$ does not fork over $M$ $\Leftrightarrow$ there is a
global $M$-invariant type $p$ with $\phi(x,a)\in p$.
\end{enumerate}
\end{fact}
\begin{proof}
(1) and (2) are by \cite[Theorem 1.1]{CheKap} and \cite[Remark 3.33]{CheKap},
(4) is from \cite{Adl}. Finally, (3) is well-known and follows from
(4). Indeed, if $\phi(x,a_{0})$ does not fork over $M$ then it is
contained in some global type $p(x)$ invariant over $M$. But then
by invariance $\{\phi(x,a_{i})\}_{i<\omega}\subseteq p(x)$, thus
does not fork over $M$. \end{proof}
\begin{defn}
Let $M$ be a small model. We say that $(\phi(x,y),q(y))$ (where
$\phi\in L(M)$ and $q$ is a partial type over $M$) is a \emph{non-forking
family over $M$} if for every $a\models q(y)$, the formula $\phi(x,a)$
does not fork over $M$.
\end{defn}
Notice that by Fact \ref{Properties of forking in NIP}(2), if $\left(\phi(x,y),q(y)\right)$
is a non-forking family, then there is some formula $\psi(y)\in q$
such that $\left(\phi(x,y),\psi(y)\right)$ is a non-forking family.
\begin{prop}
\label{prop: pqinv} Let $(\phi(x;y),q(y))$ be a non-forking family
over $M$, then there are finitely many global $M$-invariant types
$p_{1},...,p_{n-1}$ such that for every $a\models q(y)$, there is
$i<n$ with $p_{i}\vdash\phi(x;a)$. \end{prop}
\begin{proof}
Let $M\prec N$ be such that $N$ is $|M|^{+}$-saturated. 

Consider the set $X=\left\{ x\in\M\,:\,\mbox{tp}(x/N)\mbox{ is }M\mbox{-invariant}\right\} $,
it is type-definable over $N$ by $\left\{ \phi(x,a)\leftrightarrow\phi(x,b)\,:\, a,b\in N,\, a\equiv_{M}b,\,\phi\in L\right\} $.
Let $\mathcal{F}\overset{\mbox{def}}{=}\{Y\subseteq X\,:\, Y=X\cap\phi(x,a),\, a\in q(N)\}$,
and notice that the dual $\mbox{VC}$-dimension of $\mathcal{F}$
is finite, say $k$ (as $\phi(x,y)$ is $\NIP$).

Assume that for any $p<\omega$, $\mathcal{F}$ does not satisfy the
$\left(p,k\right)$-property. As by Fact \ref{Properties of forking in NIP}(2)
the set $\{\left(a_{0}...a_{k-1}\right)\,:\,\phi(x,a_{0}...a_{k-1})\mbox{ forks over }M\}$
is type-definable, by Ramsey, compactness and Fact \ref{Properties of forking in NIP}(4)
we can find an $M$-indiscernible sequence $\left(a_{i}\right)_{i<\omega}\subseteq q(N)$
such that $\bigwedge_{i<k}\phi(x,a_{i})$ forks over $M$, contradicting
Fact \ref{Properties of forking in NIP}(3) and the assumption on
$q$.

Thus $\mathcal{F}$ satisfies the $(p,k)$-property for some $p$.
Let $n$ be as given by Fact \ref{Fac: (p,k)-theorem} and define
\[
Q(x_{0},...,x_{n-1})\overset{\mbox{def}}{=}\{x_{i}\in X\}_{i<n}\cup\left\{ \bigvee_{i<n}\phi(x_{i},a)\,:\, a\in q(N)\right\} .
\]
 As every finite part of $Q$ is consistent by Fact \ref{Fac: (p,k)-theorem},
there are $b_{0}...b_{n-1}$ realizing it, take $p_{i}\overset{\mbox{def}}{=}\mbox{tp}(b_{i}/N)$.\end{proof}
\begin{rem}
If $q(x)$ is a complete type then this holds with $n=1$, just by
taking some $M$-invariant $p_{0}(x)$ containing $\phi(x,a)$.
\end{rem}
However, we cannot hope to replace invariant $\phi$-types by definable
$\phi$-types in the proposition.
\begin{example}
Let $T$ be the theory of a complete discrete binary tree with a valuation
map. Let $M_{0}$ be the prime model, and take $c$ an element of
valuation larger than $\Gamma(M_{0})$. Let $d$ be the smallest element
in $M_{0}$. Let $\phi(x;y,z)$ say ``if $z=d$, then $\mbox{val}(x)>\mbox{val}(y)$,
if $z\neq d$, then $x>y$'' (where $>$ is the order in the tree).
Let $\psi(y,z)=$ ``$z=d$''. Then $(\phi,\psi)$ is a non-forking
family over $M$, however there is no definable $\phi$-type consistent
with $\phi(x;c,d)$.\end{example}
\begin{rem}
In \cite{CotterStarchenko} it is proved that if $T$ is a $VC$-minimal
theory with unpacking and $M\models T$, then $\phi(x,a)$ does not
fork over $M$ if and only if there is a global $M$-definable type
$p(x)$ such that $\phi(x,a)\in p$. The previous example shows that
the same result cannot hold in a general $\NIP$ theory.\end{rem}
\begin{problem}
Assume $\phi(x,a)$ does not fork over $M$. Is there a formula $\psi(y)\in\tp(a/M)$
such that $\left\{ \phi(x,a)\,:\,\models\psi(a)\right\} $ is consistent
(and thus does not fork over $M$)?
\end{problem}

\section{Sufficient conditions for boundedness of $T_{\P}$}

In \cite{ExtDefI} we have demonstrated the following result.
\begin{fact}
\label{fac: boundedness implies NIP}
\begin{enumerate}
\item Let $\left(M,A\right)$ be bounded. If $M$ is $\NIP$ and $A_{\ind}$
is $\NIP$, then $\left(M,A\right)$ is $\NIP$.
\item Let $\left(M,A\right)$ be bounded and $A\prec M$. If $M$ is $\NIP$
then $\left(M,A\right)$ is $\NIP$.
\end{enumerate}
\end{fact}
However, a general sufficient condition for the boundedness of an
expansion by a predicate for $\NIP$ theories is missing. In the stable
case, a satisfactory answer is given in \cite{CaZi}. Recall:
\begin{defn}

\begin{enumerate}
\item $T$ satisfies $\nfcp$ (no finite cover property) if for any $\phi(x,y)$
there is $k<\omega$ such that for any $A$, if $\left\{ \phi(x,a)\right\} _{a\in A}$
is $k$-consistent, then it is consistent.
\item We say that $M\models T$ satisfies $\nfcp$ over $A\subset M$ if
for any $\phi(x,y,z)$ there is $k<\omega$ such that for any $A'\subseteq A$
and $b\in M$, if $\left\{ \phi(x,a,b)\right\} _{a\in A'}$ is $k$-consistent,
then it is consistent.
\end{enumerate}
\end{defn}
And then one has:
\begin{fact}
\label{CZ}Let $T$ be stable.
\begin{enumerate}
\item \cite[Proposition 2.1]{CaZi} Assume that $A\subset M\models T$ is
small and $M$ has $\nfcp$ over $A$. Then $\left(M,A\right)$ is
bounded.
\item \cite[Proposition 4.6]{CaZi} In fact, ``$\nfcp$ over $A$'' can
be relaxed to ``$A_{\ind}$ is $\nfcp$''. 
\end{enumerate}
\end{fact}
In this section we present results towards a possible generalization
for unstable $\NIP$ theories.

\subsection{Dnfcp (nfcp for definable sets of parameters)}
\begin{defn}
We say that $M$ satisfies $\dnfcp$ over $A\subseteq M$ if for any
$\phi(x,y,z)$ there is $k\in\omega$ such that: for any $b\in M$,
if $\left\{ \phi(x,a,b)\,:\, a\in A\right\} $ is $k$-consistent,
then it is consistent.
\end{defn}
We remark that $\dnfcp$ over $A$ is an elementary property of the
pair $\left(M,A\right)$.
\begin{lem}

\begin{enumerate}
\item $\nfcp$ over $A$ $\Rightarrow$ $\dnfcp$ over $A$.
\item If $T$ is stable and $M\models T$, then $\nfcp$ $\Leftrightarrow$
$\nfcp$ over $M$ $\Leftrightarrow$ $\dnfcp$ over $M$.
\end{enumerate}
\end{lem}
\begin{proof}
(1) Clear.

(2) Assume that $T$ is stable. Then $\nfcp$ and $\nfcp$ over $M$
are easily seen to be equivalent. Assume that $T$ has fcp, then by
Shelah's $\nfcp$ theorem \cite[Theorem 4.4]{MR1083551} there is
a formula $E(x,y,z)$ such that $E(x,y,c)$ is an equivalence relation
for every $c$ and for each $k\in\omega$ there is $c_{k}$ such that
$E(x,y,c_{k})$ has more than $k$, but finitely many equivalence
classes. Taking $\phi(x,y,z)=\neg E(x,y,z)$ and $M$ big enough we
see that $\left\{ \phi(x,a,c_{k})\,:\, a\in M\right\} $ is $k$-consistent,
but inconsistent.\end{proof}
\begin{lem}
If every formula of the form $\phi(x,y,z)$ with $\left|x\right|=1$
is $\dnfcp$ over $A$, then $T$ is $\dnfcp$ over $A$.\end{lem}
\begin{proof}
Assume we have proved that all formulas with $\left|x\right|\leq m$
are $\dnfcp$, and we prove it for $\left|x\right|=m+1$. So assume
that for every $n<\omega$ we have some $c_{n}\in M$ such that $\left\{ \phi(x_{0}...x_{m},a,c_{n})\right\} _{a\in A}$
is $n$-consistent, but not consistent. Let $\psi(x_{1}...x_{m},y_{i}...y_{l},z)=\exists x_{0}\bigwedge_{i\leq l}\phi(x_{0}...x_{m},y_{i},z)$,
of course still $\left\{ \psi(\bar{x},\bar{a},c_{n})\right\} _{\bar{a}\in A}$
is $\lfloor n/l\rfloor$-consistent, so consistent for $n$ large
enough by the inductive assumption. Let $b_{1}...b_{m}$ realize it.
Then consider $\Gamma=\left\{ \theta(x_{0},a,c_{n}b_{1}...b_{m})\right\} _{a\in A}$
where $\theta(x_{0},a,c_{n}b_{1}...b_{m})=\phi(x_{0}b_{1}...b_{m},a,c_{n})$.
It is $l$-consistent. Again by the inductive assumption, if $l$
was chosen large enough, there is some $b_{0}$ realizing $\Gamma$,
but then $b_{0}...b_{m}\models\left\{ \phi(x_{0}...x_{m},a,c_{n})\right\} _{a\in A}$
- a contradiction.\end{proof}
\begin{example}
$DLO$ has dnfcp over models.
\end{example}
The following criterion for boundedness follows from the proof of
\cite{CaZi}. 
\begin{thm}
\label{fac: dnfcp -> boundedness} Let $A\subset M$ be small and
uniformly stably embedded. Assume that $M$ has $\dnfcp$ over $A$.
Then $\left(M,A\right)$ is bounded.
\end{thm}
The problem with $\dnfcp$ is that it does not seem possible to conclude
$\dnfcp$ over $A$ from properties of the induced structure on $A$.
To remedy this, we introduce a weaker variant with separated variables.
\begin{defn}
We say that $M$ satisfies $\dnfcp'$ over $A\subseteq M$ if for
any $\phi(x,y)$ and $\psi(y,z)$, there is $k<\omega$ such that
for any $b\in M$, if $\{\phi(x,a):a\in\psi(A,b)\}$ is $k$-consistent,
then it is consistent. We say that $T$ has $\dnfcp'$ if for any
$M\prec N$, $N$ has $\dnfcp'$ over $M$.\end{defn}
\begin{rem}
Let $(M,A)$ be a pair, and assume that $A$ is small and $A_{\ind}$
is saturated. Then if formulas are bounded, $M$ has $\dnfcp'$ over
$A$. \end{rem}
\begin{proof}
By assumption $\exists y\forall a\in\P,\psi(a;z)\rightarrow\phi(a;y)$
is equivalent to a bounded formula $\theta(z)$, for any $\phi$ and
$\psi$. If $\dnfcp'$ does not hold, then there is a consistent bounded
type satisfying $\neg\theta(z)$ and for all $n$, $\forall a_{1},...,a_{n}\in\P\exists y,\bigwedge\psi(a_{i};z)\rightarrow\phi(a_{i};y)$.
As $A_{ind}$ is saturated, it is resplendent, and we can find a type
over $A$ which satisfies this bounded type. By smallness of $A$
in $M$, this type is realized by some $c\in M$. Then again by smallness,
there is $b\in M$ such that $\psi(a;c)\rightarrow\phi(a;b)$ for
all $a\in A$. This contradicts the hypothesis on $\theta$.
\end{proof}
We can now prove some preservation result.
\begin{lem}
\label{lem:dnfcp goes up}Let $T$ be NIP, $A\subseteq M\models T$
and assume that $Th\left(A_{\ind(L_{P})}\right)$ has $\dnfcp'$.
Then $M$ has $\dnfcp'$ over $A$.\end{lem}
\begin{proof}
Let $\phi(x,y)$ and $\psi(y,b)$ be given. Let $\theta_{\phi}(y,s)$
be a uniform honest definition for $\phi$ and $\theta_{\psi}(y,t)$
a uniform honest definition for $\psi$ (by Theorem \ref{thm: Uniform honest definitions}).
Let $\left(M',A'\right)\succ\left(M,A\right)$ be a sufficiently saturated
elementary extension, then naturally $A'_{ind(L_{P})}\succ A_{ind(L_{P})}$.
There is $c_{\psi}\in A'$ such that $\psi(A,b)=\theta_{\psi}(A,c_{\psi})$.

Let $\chi(s)$ be the formula $\exists d\forall y\in P\theta_{\phi}(y,s)\rightarrow\phi(d,y)$
and let $k\in\omega$ be as given for $\theta_{\phi}(y,s)\land\chi(s)$,
$\theta_{\psi}(y,t)$ by $\dnfcp'$ of $A_{ind(L_{P})}$ for it. Assume
that $\left\{ \phi(x,a)\,:\, a\in\psi(A,b)\right\} $ is $k$-consistent,
then $\left\{ \theta_{\phi}(a,s)\land\chi(s)\,:\, a\in\theta_{\psi}(A,c_{\psi})\right\} $
is $k$-consistent (let $d\models\left\{ \phi(x,a_{i})\right\} _{i<k}$,
and choose $c_{\phi}\in A$ such that $\left\{ a_{i}\right\} _{i<k}\subseteq\theta_{\phi}(A,c_{\phi})\subseteq\phi(d,A)$).
As $A_{ind(L_{P})}$ is $\dnfcp'$, we conclude that it is consistent.
In particular, for any $n\in\omega$ and $a_{0},...,a_{n}\in\theta_{\psi}(A,c_{\psi})=\psi(A,b)$,
there is $c_{\phi}\in A$ such that $\bigwedge_{i<n}\theta_{\phi}(a_{i},c_{\phi})\land\chi(c_{\phi})$,
thus unwinding there is some $d\models\left\{ \phi(x,a_{i})\right\} _{i<n}$.
\end{proof}

\subsection{Boundedness of the pair for distal theories}

We now aim at giving an analog of Theorem \ref{CZ} for distal theories
and stably embedded predicates.

First, we improve Lemma \ref{lem:dnfcp goes up}.
\begin{lem}
\label{lem:dnfcp goes up2}Let $T$ be distal, $A\subseteq M\models T$
and assume that $Th\left(A_{\ind(L)}\right)$ has $\dnfcp'$. Then
$M$ has $\dnfcp'$ over $A$.\end{lem}
\begin{proof}
Follow the proof of Lemma \ref{lem:dnfcp goes up}, except that we
define $\chi(s)$ as $\exists x\forall y\theta_{\phi}(y,s)\rightarrow\phi(d,y)$,
which we can by strong honest definitions (Lemma \ref{thm: distal =00003D uniform strong honest definitions}).
\end{proof}
Let $A_{0}$ be a small subset of $M_{0}$, and take a $|T|^{+}$-saturated
$\left(M,A\right)\succ\left(M_{0},A_{0}\right)$.
\begin{lem}
\label{prop: toward boundedness in distal} Assume that $T$ is distal
and $M$ has $\dnfcp'$ over $A$. Let $a\in M$,$\zeta(x,y)\in L$
and $q(y)\in S(A)$ be an $a$-definable type. Then the following
are equivalent:
\begin{enumerate}
\item There is $b\models q$ in $\M$ such that $\models\zeta(a,b)$.
\item There is $b\models q$ in $M$ such that $\models\zeta(a,b)$.
\end{enumerate}
\end{lem}
\begin{proof}
By $L_{\P}$-saturation of $(M,A)$ and definability of $q(y)$ over
$a$, it is enough to find such a $b$ realizing the $\phi(y,z)$-part
of $q(y)$. Assume that it is definable by $d_{\phi}(z,a)$. Let $\theta(y,t)$
be given by Proposition \ref{prop: isolated extension} for $\phi$,
and let $d_{\theta}(t,a)$ define the $\theta$-part of $q$. By $\dnfcp'$,
the fact that $d_{\phi}(z,a),d_{\theta}(t,a)$ define a consistent
$\phi,\theta$-type $q_{a}$ over $\P$ is expressible by a bounded
formula $\psi_{1}(a)$ saying: 

\[
\forall z_{1}...z_{n}\in\P\forall t_{1}...t_{n}\in\P\,\exists y\left(\bigwedge_{i\leq n}\phi(y,z)\leftrightarrow d_{\phi}(z,a)\,\land\,\bigwedge_{i\leq n}\theta(y,t)\leftrightarrow d_{\theta}(t,a)\right),
\]

where $n$ is given by $\dnfcp'$ for $\phi'(y,z_{1}z_{2}t_{1}t_{2})=\phi(y,z_{1})\land\neg\phi(y,z_{2})\land\theta(y,t_{1})\land\neg\theta(y,t_{2})$
and $\psi'(z_{1}z_{2}t_{1}t_{2},\alpha)=d_{\phi}(z_{1},\alpha)\land\neg d_{\phi}(z_{2},\alpha)\land d_{\theta}(t_{1},\alpha)\land\neg d_{\theta}(t_{2},\alpha)$.

Observe that for any $d\in d_{\theta}(A,a)$, $M\models\exists b\,\theta(b,d)\land\zeta(a,b)$
(as $q(y)\land\zeta(a,y)$ is consistent). It can be expressed by
a bounded formula $\psi_{2}(a)$.

Let $a_{0}\in M_{0}$ be such that $\left(M_{0},A_{0}\right)\models\psi_{1}(a_{0})\land\psi_{2}(a_{0})$.
Assume that there is a finite $C\subseteq A_{0}$ such that $q_{a_{0}}(y)|_{C}\land\zeta(a_{0},y)$
is inconsistent. Let $d\in d_{\theta}(A_{0},a_{0})$ be as given by
Theorem \ref{thm: distal =00003D uniform strong honest definitions}.
Then find some $b\in M_{0}$ such that $\models\theta(b,d)\land\zeta(a_{0},b)$
(by $\psi_{2}(a_{0})$). By the hypothesis on $\theta$, we have $b\models q_{a_{0}}|C$
-- a contradiction. 

So $q_{a_{0}}(y)\land\zeta(a_{0},y)$ is consistent, and it follows
by smallness of $A_{0}$ in $M_{0}$ that $\left(M_{0},A_{0}\right)\models\forall x\,\psi_{1}(x)\land\psi_{2}(x)\rightarrow\exists b\models q_{x}(y)\land\zeta(x,y)$.
It follows that $\left(M,A\right)$ satisfies the same sentence, and
unwinding we conclude.\end{proof}
\begin{thm}
Let $T$ be distal, $A\subseteq M$ is small and (uniformly) stably
embedded, and $A_{\ind}$ has $\dnfcp'$. Then $T_{\P}$ is bounded.\end{thm}
\begin{proof}
By Lemma \ref{lem:dnfcp goes up2}, $M$ has $\dnfcp'$ over $A$.
Take $(M,A)$ a $\left|T\right|^{+}$-saturated elementary extension
of the pair. Let $a,a'\in M$ be such that $A_{[a]}\equiv A_{[a']}$.
We have to show that $\tp_{L_{\P}}(a)=\tp_{L_{\P}}(a')$. We do a
back-and-forth. Take $b\in M$.

Case 1: $b\in A$. As $A_{[a]}\equiv A_{[a']}$, by $L_{\P}$-saturation
we can find $b'\in P$ such that $A_{[ab]}\equiv A_{[a'b']}$.

Case 2: $b\in M\setminus A$. By stable embeddedness and Case 1, we
may assume that $\tp(ab/A)$ is $a$-definable. It is enough to find
$b'\in M\setminus A$ such that $\tp(b',a')=\tp(b,a)$ and $\tp(ab'/A)$
is defined over $a'$ the same way $\tp(ab/A)$ is over $a$. Now
the previous lemma (and saturation) applies and gives such a $b'$. 
\end{proof}

\section{Naming indiscernible sequences, again}

We recall briefly the story of the question. In \cite{BB} Baldwin
and Benedikt had established the following.
\begin{fact}
\label{fac: BB boundedness} Let $T$ be NIP. Let $I\subset M$ be
a small indiscernible sequence indexed by a dense complete linear
order. Then $\Th\left(M,I\right)$ is bounded and the $L_{\P}$-induced
structure on $I$ is just the linear order.
\end{fact}
We have demonstrated (\cite[Proposition 3.2]{ExtDefI}) that in this
case $\left(M,I\right)$ is still $\NIP$. In this section we are
going to complement the picture by resolving some of the remaining
questions: naming a small indiscernible sequence of \emph{arbitrary}
order type preserves $\NIP$, while naming a large indiscernible sequence
may create $\IP$.

\subsection{Naming an arbitrary small indiscernible sequence}
\begin{lem}
\label{lem: small (M,I) equiv small (N,I)} Let $I$ be small in $M$
and $N\succ M$ such that $I$ is small in $N$. Then $\left(M,I\right)$
and $\left(N,I\right)$ are elementary equivalent.\end{lem}
\begin{proof}
We do a back and forth starting with the identity mapping from $I$
to $I$, and inductively choosing $A=\{a_{i}\}_{i<\omega}\subset M$
and $B=\left\{ b_{i}\right\} _{i<\omega}\subset N$ such that $\mbox{tp}_{L}(AI)=\mbox{tp}_{L}(BI)$.
Assume we have chosen $\{a_{m}b_{m}:m<n\}$ and we pick $a_{n}\in M$.
Consider $p(x,AI)=\mbox{tp}_{L}(a_{n}/AI)$. By the inductive assumption,
$p(x,BI)$ is consistent. Let $b_{n}\in N$ realize it (possible by
smallness). In the end, in particular, $AI\equiv^{qf-L_{P}}BI$.
\end{proof}
Assume that $D$ is an $L$-definable set which is uniformly stably
embedded in the sense of $T$ (and $T$ eliminates quantifiers in
a relational language $L$), let $\P$ name a subset of $D$. Now
let $(N,P)$ be a saturated model of the pair.

A formula is $D$-bounded if it is equivalent to one of the form $\psi(\bar{x})=Q_{1}z_{1}\in D...Q_{n}z_{n}\in D\,\bigvee_{i<m}\phi_{i}(\bar{x},\bar{z})\land\chi_{i}(\bar{x},\bar{z})$,
where $\phi_{i}(\bar{x},\bar{z})$ is a qf-$L$-formula and $\chi_{i}(\bar{x},\bar{z})$
is a qf-$\P$-formula (follows from the relationality of $L$).
\begin{lem}
\label{lem: formulas are D-bounded} Let $a,a'\in N$ have the same
$D$-bounded type, then $a\equiv^{L_{\mathbf{P}}}a'$.\end{lem}
\begin{proof}
We do a back-and -forth. Assume that $a\equiv^{L^{D-bdd}}a'$, and
let $b\in N$ be arbitrary.

Case 1. $b\in D$: Consider $p(x,a)=\tp_{L^{D-bdd}}(ba)$. For any
finite $p_{0}(x,a)\subseteq p(x,a)$ we have $\models\exists x\in D\, p_{0}(x,a)$,
which is a $D$-bounded formula, thus $\models\exists x\in D\, p_{0}(x,a')$,
and by saturation of $N$ there is $b'\in D$ satisfying $ab\equiv^{L^{D-bdd}}a'b'$.

Case 2. $b\notin D$: Possibly adding some points in $D$ using (1),
we may assume that $\tp_{L}(ab/D)$ is $L$-definable over $c=a\cap D$.
Take some $b'\in N$ such that $ab\equiv^{L}a'b'$, then $\tp_{L}(a'b'/D)$
is $L$-definable over $c'=a'\cap D$ using the same formulas. We
want to check that $ab\equiv^{L^{D-bdd}}a'b'$. Let $\psi(\bar{x})$
be a $D$-bounded formula, say $\psi(\bar{x})=Q_{1}z_{1}\in D...Q_{n}z_{n}\in D\bigvee_{i<m}\phi_{i}(\bar{x},\bar{z})\land\chi_{i}(\bar{x},\bar{z})$.
Then we have:\\
$\models Q_{1}x_{1}\in D...Q_{n}x_{n}\in D\bigvee_{i<m}\phi_{i}(ab,\bar{x})\land\chi_{i}(ab,\bar{x})$
$\Leftrightarrow$ $\models\bar{Q}\bar{x}\in D\bigvee_{i<m}d_{\phi_{i}}(c,\bar{x})\land\chi_{i}^{'}(\bar{x})$
(as we know the truth values of $\P(x)$ on $ab$) $\Leftrightarrow$
$\models\bar{Q}\bar{x}\in D\bigvee_{i<m}d_{\phi_{i}}(c',\bar{x})\land\chi_{i}^{'}(\bar{x})$
(as $c\equiv^{L_{\P}^{D-bdd}}c'$) $\Leftrightarrow$ $\models Q_{1}x_{1}...Q_{n}x_{n}\bigvee_{i<m}\phi_{i}(a'b',\bar{x})\land\chi_{i}(a'b',\bar{x})$
(as the truth values of $\P(x)$ on $a'b'$ are the same by the choice
of $b'$ and assumption on $a'$). \end{proof}
\begin{lem}
\label{lem: indiscernibles, P-boundedness} Assume that $\Th(D_{\ind},P)$
is bounded. Then $\Th(M,P)$ is bounded.\end{lem}
\begin{proof}
Let $\left(N,P\right)$ be saturated. Assume that $P_{\left[a\right]}\equiv P_{[a']}$
and let $b$ be given.

If $b\in D$, then we find a $b'\in D$ such that $P_{[ab]}\equiv P_{[a'b']}$
by the assumption that $\left(D,P\right)$ is bounded and saturation. 

If $b\notin D$, then we take the same $b'$ as in (2) of the previous
lemma and conclude that $bb'\equiv^{L_{P}^{D-bdd}}aa'$ in the same
way (using that $c\equiv^{L_{P}^{P-bdd}}c'$ $\Rightarrow$$c\equiv^{L_{P}^{D-bdd}}c'$),
which is sufficient (clearly, if two tuples have the same $D$-bounded
$L_{\P}$-type, then they have the same $\P$-bounded $L_{\P}$-type). \end{proof}
\begin{lem}
\label{lem: subpair is NIP} In the situation as above, if $T$ is
$\NIP$ and $(D,P)$ with the induced quantifier-free structure is
$\NIP$, then $T_{P}$ is $\NIP$.\end{lem}
\begin{proof}
As $D_{\mbox{ind}\left(L_{P}^{qf}\right)}$ is $\NIP$, it follows
that $D_{\mbox{ind}\left(L_{P}^{D-bdd}\right)}$ is $\NIP$. Conclude
as in Corollary 2.5 in {[}CS{]} (and even easier as $D$ is actually
stably embedded).\end{proof}
\begin{thm}
Let $\left(M,I\right)$ be small and $M$ be $NIP$. Then $(M,I)$
is $NIP$.\end{thm}
\begin{proof}
Let $\left(M,I\right)$ be small. By Lemma \ref{lem: small (M,I) equiv small (N,I)}
we may assume that $M$ is $\left(2^{\left|I\right|}\right)^{+}$-saturated.
Let $I\subseteq J\subset M$, where $J$ is a dense complete indiscernible
sequence such that $\left(M,J\right)$ is still small. Name $J$ by
$D$, and let $T'$ be a Morleyzation of $T_{D}$. Then by Fact \ref{fac: BB boundedness},
$T'$ is $NIP$ and $D$ is stably embedded. Thus formulas in $T'_{\P}$
are $D$-bounded by Lemma \ref{lem: formulas are D-bounded}. It is
easy to check directly that $\left(J_{\ind},I\right)$ is bounded,
thus $T'_{\P}$ is $\P$-bounded by Lemma \ref{lem: indiscernibles, P-boundedness}.
Conclude by Fact \ref{fac: boundedness implies NIP} (as the structure
induced on $I$ is still $\NIP$).
\end{proof}

\subsection{Large indiscernible sequence producing $\IP$}

Take $L=\{<,E\}$ and $T$ saying that $<$ is $DLO$ and $E$ is
an equivalence relation with infinitely many classes, all of which
are dense co-dense with respect to $<$. It is easy to check by back-and-forth
that this theory eliminates quantifiers and that it is $\NIP$. Let
$M/E$ denote the imaginary sort of $E$-equivalence classes.

Let $D$ be an equivalence class, pick some $x_{0}\in M$ outside
of it and take $\P$ to name $D\cap(-\infty,x_{0})$. Consider the
formula 
\[
\phi(x)=\exists y\forall s<y\exists t\in\P,yEx\wedge s<t<y\wedge(\neg\exists u>y,u\in\P).
\]
Then $\phi(x)$ picks out exactly the points equivalent to $x_{0}$.
Easily, that formula is not equivalent to a $D$-bounded one (simply
because all imaginary elements of equivalence classes different from
$D$ have exactly the same type over $D$).

Now consider the following formula:

\[
S(x_{1},x_{2})=\exists y_{1},y_{2},y_{1}Ex_{1}\wedge y_{2}Ex_{2}\wedge L_{0}(y_{1})\wedge R_{0}(y_{2})\wedge(\forall y_{1}<z<y_{2},\neg\P(z))
\]

where $L_{0}(y)=\exists t\in\P\forall s\in\P,t<y\wedge(s>t\rightarrow y<s)$
and same for $R_{0}(y)$, but reversing the inequalities.
\begin{claim}

\begin{enumerate}
\item Let $D$ be an equivalence class. Then any increasing sequence contained
in $D$ is indiscernible.
\item Let $G$ be an arbitrary countable graph. Then we can choose $P\subseteq D$
such that $(M/E,S)\cong G$.
\end{enumerate}
\end{claim}
\begin{proof}
(1) is immediate by the quantifier elimination.

(2) By induction, for every edge $e_{1}e_{2}\in(M/E)^{2}$ that we
want to put, chose a pair of representatives $a_{1},a_{2}\in\mathbb{Q}$
such that the interval $(a_{1},a_{2})$ is disjoint from all the previously
chosen intervals. Let $P$ name the set of points in $D$ outside
of the union of these intervals.
\end{proof}
In particular we can choose $\P$ so that $T_{\P}$ interprets the
random graph.
\begin{rem}
We also observe that naming two small indiscernible sequences at once
can create $\IP$. This time we name sequences which satisfy $\neg xEy$
for any two points $x$ and $y$ in them. So pick any small $I_{0}$.
Let $A=A[I_{0}]=\{t\in M/E:\exists x\in I_{0},xEt\}$. Then $A$ gets
an order $<_{0}$ form $I_{0}$ induced by $<$. Fix $<_{1}$ any
other order on $A$. Then we can find another sequence $I_{1}$ such
that $A[I_{1}]=A$ and the order induced on $A$ by $I_{1}$ is $<_{1}$.
With two linear orders, we can code pseudo-finite arithmetic as in
\cite{ShSi}. In particular we have IP.
\end{rem}

\section{Models with definable types}

Classically,
\begin{fact}
$T$ is stable $\Leftrightarrow$ for every $M\models T$, $\left|S(M)\right|\leq\left|M\right|^{|T|}$
$\Leftrightarrow$ for every $M\models T$, all types over it are
definable $\Leftrightarrow$ there is a saturated $M\models T$ with
all types over it definable.
\end{fact}
We start by observing that if $T$ is $\NIP$, then it has models
of arbitrary size with few types over them.
\begin{prop}
Let $T$ be $\NIP$. For any $\kappa\geq\left|T\right|$ there is
a model $M$ with $|M|=\kappa$ such that $|S(A)|\leq|A|^{|T|}$ for
every $A\subseteq M$. \end{prop}
\begin{proof}
If $T$ is stable then every model of size $\kappa$ works. Otherwise
assume $T$ is unstable and let $I=(a_{\alpha})_{\alpha<\kappa}$
be linearly ordered by $<(x,y)\in L$. Let $T^{\mbox{Sk}}$ be a Skolemization
of $T$, and let $M=\mbox{Sk}(I)$, $\left|M\right|\leq\kappa+\left|T\right|$. 

We show that $S^{L}(M)\leq\kappa^{|T|}$. Consider 
\[
\widetilde{L}:=\{\phi(x,f(\bar{y}))\,:\,\phi\in L\,\mbox{ and \ensuremath{f}is an \ensuremath{L^{\mbox{Sk}}}-definable function}\}.
\]
Notice that every $\psi(x,y)\in\widetilde{L}$ is $NIP$. But then
(by Remark \ref{rem: types over complete sequence are definable})
for every $\psi\in\widetilde{L}$, every $\psi$-type over $I$ is
$<$-definable, in particular $|S^{\widetilde{L}}(I)|\le|I|^{|T|}$. 

Given $p,q\in S^{L}(M)$ choose some $p',q'\in S^{\widetilde{L}}(M)$
with $p\subseteq p',q\subseteq q'$. It is easy to see that $p'|_{I}=q'|_{I}$
$\Rightarrow$ $p=q$ (for any $a\in M$ and $\phi\in L$ we have
$\phi(x,a)\in p$ $\iff\phi(x,f(\bar{b}))\in p'|_{I}$ where $\bar{b}\subseteq I$
and $f(\bar{b})=a$), thus $|S^{L}(M)|\leq|S^{\widetilde{L}}(I)|\leq\kappa^{|T|}$.\end{proof}
\begin{rem}
Slightly elaborating on the argument, we may construct such an $M$
which is in addition gross ($M$ is called \textit{gross} if every
infinite subset definable with parameters from $M$ is of cardinality
$|M|$, see \cite{MR2054800}).
\end{rem}
In general one cannot find a model such that all types over it are
definable (for example, take $RCF$ and add a new constant for an
infinitesimal). However, some interesting $NIP$ theories have models
with all types over them definable.
\begin{example}

\begin{enumerate}
\item $\mathbb{R}$ as a model of $RCF$ (and this is the only model of
$RCF$ with all types definable).
\item In $ACVF$ there are arbitrary large models with all types definable
(maximally complete fields with $\mathbb{R}$ as a value group).
\item $(\mathbb{Z},+,<)$ is a model of Presburger arithmetic with all types
definable (but there are no larger models). 
\item $\left(\mathbb{Q}_{p},+,\times,0,1\right)$ (by \cite{MR953003}).
\end{enumerate}
\end{example}
When looking at a particular example, it is usually much easier to
check that $1$-types are definable, rather than $n$-types, and one
can ask if this is actually the same thing.
\begin{defn}
Let $A$ be a set. We say that it is \emph{$\left(n,m\right)$-stably
embedded} if every subset of $A^{n}$ which can be defined as $\phi(A,a)$
with $\left|a\right|\leq m$, can actually be defined as $\psi(A,b)$
with $b\in A$. We say that it is \emph{uniformly} $\left(n,m\right)$-stably
embedded if $\psi$ can be chosen depending just of $\phi$ (and not
on $a$). A compactness argument shows that for a definable set $A$,
it is $\left(n,m\right)$-stably embedded if and only if it is uniformly
$\left(n,m\right)$-stably embedded. Obviously, $\left(\infty,n\right)$-stable
embeddedness is equivalent to definability of $n$-types over $A$.
\end{defn}
Of course, $\left(n,m\right)$-stable embeddedness implies $\left(n',m'\right)$-stable
embeddedness for $n'\leq n,m'\leq m$. 
\begin{prop}
\label{prop: (infty,1) -> (1,infty)} Let $T$ be $\NIP$ and assume
that $M$ is $(\infty,n)$-stably embedded. Then it is $\left(n,\infty\right)$-stably
embedded.\end{prop}
\begin{proof}
By definability, every type $p\in S_{n}(M)$ has a unique heir.

Claim 1: If $p\in S_{n}(M)$ has a unique heir, then it has a unique
coheir.

Let $p'(x)$ be the unique global heir of $p$. Let $p_{1}(x)$ be
a global coheir of $p$, and $(a_{i})_{i<\omega}$ a Morley sequence
in it over $M$. Given $\bar{m}\in M$ and noticing that $tp(a_{0}/a_{1}...a_{n}M)$
is an heir over $M$ (so is contained in a global heir as $M\models T$)
we have that $\models\phi(a_{0},...,a_{n},\bar{m})$ if and only if
$\phi(x,a_{1}...a_{n}\bar{m})\in p'(x)$. Thus by Fact \ref{fac: MS determines invariant type},
$p$ has a unique global coheir.

Claim 2: Every $p\in S_{n}(A)$ has a unique coheir $\Leftrightarrow$
$A$ is $(n,\infty)$-stably embedded.

$\Rightarrow$: Let $\phi(x,c)\in L(\mathbb{M})$ and consider $p(x)\in S_{n}(A)$
finitely satisfiable in $\phi(x,c)\cap A$. If it was finitely satisfiable
in $\neg\phi(x,c)\cap A$ as well, then $p$ would have two coheirs,
thus there is some $\psi_{p}(x)\in p(x)$ with $\psi_{p}(x)\rightarrow^{A}\phi(x,c)$.
By compactness we have $\bigvee\psi_{p_{i}}(x)\leftrightarrow^{A}\phi(x,c)$
for finitely many $p_{i}$'s. 

$\Leftarrow$: Let $p_{1},p_{2}$ be two global coheirs of $p\in S_{n}(A)$,
and assume that $\phi(x,a)\in p_{1},\neg\psi(x,a)\in p_{2}$. Let
$\psi(x)\in L(A)$ be such that $\psi(A^{n})=\phi(A^{n},a)$. It follows
that $\psi(x)\in p$. But this implies that $p_{2}$ cannot be a coheir
as $\psi(x)\land\neg\phi(x,a)$ is not realized in $A$.
\end{proof}
And so it is natural to ask whether $(\infty,1)$-stable embeddedness
of $M$ implies $\left(\infty,n\right)$-stable embeddedness. The
answer is yes in stable theories, for the obvious reason, and yes
in $o$-minimal theories, where by a theorem of Marker and Steinhorn
\cite{MR1264974}, $\left(1,1\right)\rightarrow\left(\infty,\infty\right)$
for models. However, we show in the next section that this is not
true in $\NIP$ theories in general. The question remains open for
$C$-minimal theories.

\subsection{Example of $\left(\infty,1\right)\not\rightarrow\left(\infty,m\right)$}

\subsubsection{General construction}

Start with a theory $T$ in a language $L$ containing an equivalence
relation $E(x,y)$. Assume $T$ has a model $M_{0}$ composed of $\omega$-many
$E$-equivalence classes, each one finite of increasing sizes. So
that any model $M$ of $T$ contains $M_{0}$ as a sub-model and all
the $E$-classes disjoint from $M_{0}$ are infinite.

We consider the language $L'$ defined as follows:
\begin{itemize}
\item For each relation $R(x_{1},...,x_{n})$ in $L$, $L'$ contains a
relation $R'(x_{1},y_{1},...,x_{n},y_{n})$.
\item Also $L'$ contains an equivalence relation $\tilde{E}(u,v)$, a binary
relation $S(u,v)$ and a quaternary relation $U(u_{1},v_{1},u_{2},v_{2})$.
The relation $S$ will code a graph and $U$ will be an equivalence
relation on $S$-edges. 
\end{itemize}
We build an $L'$ structure $N_{0}$ as follows:

$N_{0}$ has $\omega$-many $\tilde{E}$-equivalence classes, corresponding
to the $E$-equivalence classes of $M_{0}$. Let $\mathfrak{e}$ be
an $E$-class, and let $n$ be its size. Then the corresponding $\tilde{E}$
class $\tilde{\mathfrak{e}}$ in $N_{0}$ is a finite regular graph,
with $S$ as the edge relation, of degree $n$ (every vertex has degree
$n$) and with no cycles of length $\leq n$ (such graphs exist, see
e.g. \cite[III.1, Theorem 1.4']{MR506522}). The predicate $U$ is
interpreted as an equivalence relation between edges so that every
vertex is adjacent to exactly one edge from each equivalence class.
We fix a bijection $\pi$ between $U$-equivalence classes and elements
of the $E$-class $\mathfrak{e}$. This being done, for each relation
$R(x_{1},...,x_{n})$ we say that $R'(x_{1},y_{1},...,x_{n},y_{n})$
holds in $N_{0}$ if $\bigwedge_{i\leq n}S(x_{i},y_{i})$ and if $R(\pi(x_{1},y_{1}),...,\pi(x_{n},y_{n}))$
holds in $M_{0}$.

Note that any model of $T'=Th(N_{0})$ contains $N_{0}$ as submodel
and its $\tilde{E}$-classes not in $N_{0}$ are infinite and composed
of disjoint unions of trees with infinite branching. So the graph
structure does not interact in any way with the structure coming from
the $R'$ relations.

Given a model of $T'$ we can recover a model of $M_{0}$ by looking
at $U$-equivalence classes and we obtain in this way every model
of $T$. So there are at least as many 2-types over $N_{0}$ as there
are 1-types over $M_{0}$. However, the non-realized 1-types over
$N_{0}$ correspond to imaginary types of non-realized $E$-classes
over $M_{0}$. See below. \\

Assume that $L$ contains a constant for every element of $M_{0}$.
Let $N\models T'$ and denote by $M$ the model of $T$ which we get
from $N$. We build a language $L''\supset L'$:
\begin{itemize}
\item We add a constant for every element of $N_{0}$.
\item For every $n\in\omega$, we add a relation $d_{n}(u,v)$ which holds
if and only if $u$ and $v$ are at distance $n$ (in the sense of
the shortest path in graph $S(u,v)$).
\item For every $\emptyset$-definable set $\phi(x_{1},...,x_{n},y_{1},...,y_{m})$
of $M_{0}$ which is $E$-congruent with respect to the variables
$x_{i}$ (\textit{i.e.}, for $a_{i}Ea'_{i}$ and $b_{i}$'s, we have
\\
$\phi(a_{1},...,a_{n},b_{1},...,b_{m})\leftrightarrow\phi(a'_{1},...,a'_{n},b_{1},...,b_{m})$)
we add a predicate \\
$W_{\phi}(x_{1},...,x_{n},y_{1},z_{1},...,y_{m},z_{m})$ which we
interpret as:\\
$N\models W_{\phi}(a_{1},...,a_{n},b_{1},c_{1},...,b_{m},c_{m})$
if and only if $\bigwedge_{i\leq m}S(b_{i},c_{i})$ and for some $e_{1},...,e_{n}\in M$
with $e_{i}$ in the $E$-class corresponding to the $\tilde{E}$-class
of $a_{i}$, we have $M\models\phi(e_{1},...,e_{n},\pi(b_{1},c_{1}),...,\pi(b_{m},c_{m}))$.\end{itemize}
\begin{claim}
If $T$ eliminates quantifiers in $L$, then $T'$ eliminates quantifiers
in $L''$. \end{claim}
\begin{proof}
By easy back-and-forth. \end{proof}
\begin{cor}
If $T$ is NIP, then $T'$ is NIP. 
\end{cor}

\begin{cor}
Assume that all (imaginary) types of a new $E$ class in $M_{0}$
are definable, then all 1-types over $N_{0}$ are definable. 
\end{cor}

\subsubsection{An example of $M_{0}$ with $\NIP$}

Let $L_{0}=\{\leq,E\}$. We build an $L_{0}$-structure $M_{0}$ as
follows:
\begin{itemize}
\item The reduct to $\leq$ is a binary tree with a root (every element
has exactly two immediate successors, there is a unique element with
no predecessor). The tree is of height $\omega$, so every element
is at finite distance from the root.
\item Two elements are $E$-equivalent if they are at the same distance
from the root. 
\end{itemize}
This theory eliminates quantifiers in the language $L$ obtained from
$L_{0}$ by adding a constant for every element of $M_{0}$, a binary
function symbol $\wedge$ interpreted as $x\wedge y$ is the maximal
element $z$ such that $z\leq x$ and $z\leq y$ and for each $n$
a predicate $d_{n}(x,y)$ saying that the difference between the heights
of $x$ and $y$ is $n$. Note that those predicates are $E$-congruent.

Clearly, $M_{0}$ is NIP, there is a unique imaginary type of a new
$E$-class over $M_{0}$ and this type is definable. However, not
all types over $M_{0}$ are definable.

So we obtain the required counter-example. 
\begin{rem}
Together with Proposition \ref{prop: (infty,1) -> (1,infty)} it follows
that also $\left(1,\infty\right)\not\rightarrow(n,\infty)$ in a general
$\NIP$ theory. Another example due to Hrushovski witnessing this
is presented in Pillay \cite{MR2830421} -- a proper dense elementary
pair of $ACVF$'s $F_{1}\prec F_{2}$ with the same residue field
and value group. Then $F_{1}$ is $(1,\infty)$-stably embedded in
$F_{2}$, but if $a\in F_{2}\setminus F_{1}$, then the function taking
$x\in F_{1}$ to $v(x-a)$ is not $F_{1}$-definable. 
\end{rem}

\subsection{Some positive results}

In \cite{MR2830421} Pillay had established the following.
\begin{fact}
Let $A$ be a definable subset of $M$. Assume that $A_{\ind}$ is
rosy, $M$ is $\NIP$ over $A$ and $A$ is $\left(1,\infty\right)$-stably
embedded. Then $A$ is stably embedded.
\end{fact}
In fact, one can replace the definable set $A$ with a model, at the
price of requiring that $\left(1,\infty\right)$-stable embeddedness
is uniform. We explain briefly how to modify Pillay's argument.
\begin{thm}
Let $A\preceq M$. Assume that $A_{\ind}$ is rosy, $M$ is $\NIP$
over $A$ and $A$ is \textbf{uniformly} $\left(1,\infty\right)$-stably
embedded. Then $A$ is uniformly stably embedded.\end{thm}
\begin{proof}
Assume that $A\preceq M$ is a counterexample to the theorem. We consider
$(M,A)$ as a pair with $\P$ naming $A$. As $A$ is a model, it
follows that $A_{\ind}$ eliminates quantifiers, thus \emph{every
set definable in $A_{\ind}$ is given by the trace of an $L$-formula}.
As there are two languages $L$ and $L_{\P}$ around, we make a terminology
clarification: induced structure is always meant to be with respect
to $L$ formulas, and $(n,m)$-stable embeddedness always means that
sets externally definable by $L$-formulas are internally definable
by $L$-formulas.
\begin{claim*}
We may assume that $\left(M,A\right)$ is saturated (as a pair in
the $L_{\P}$ language).\end{claim*}
\begin{proof}
Just let $\left(N,B\right)\succ\left(M,A\right)$ be a saturated extension.
Of course, $A$ is uniformly $\left(n,\infty\right)$-stably embedded
in $M$ if and only if $B$ is uniformly $\left(n,\infty\right)$-stably
embedded in $N$. Notice that $B_{\ind}\succ A_{\ind}$, thus $B_{\ind}$
is rosy. Finally, $N$ is still $\NIP$ over $B$ with respect to
$L$-formulas.
\end{proof}

\begin{claim*}
Let $f:\, A\to Z$ be an $L(M)$-definable function (namely the trace
on $A$ of an $L(M)$-definable relation which happens to define a
function on $A$), where $Z$ is some sort in $A_{\ind}^{\eq}$ .
Then there is an $L(A)$-definable relation $R(x,y)$ and $k<\omega$
such that $(M,A)\models\forall x\in\P\,\left(R(x,f(x))\land\exists^{\leq k}y\in\P\, R(x,y)\right)$.\end{claim*}
\begin{proof}
Let the graph of $f$ be defined by $f(x,y,e)\in L(M)$. Let $\kappa$
be large enough. Working entirely in $A_{\ind}$, assume that we could
choose $\left(a_{i}b_{i}\right)_{i<\kappa}$ in $A$ such that $b_{i}=f(a_{i})$
and $b_{i}\notin acl_{L}\left(\left(a_{j}b_{j}\right)_{j<i}a_{i}\right)$
for all $i$. Following Pillay's proof of \cite[Lemma 3.2]{MR2830421}
and using saturation of $A_{\ind}$, we may assume that $\left(a_{i}b_{i}\right)$
is $L$-indiscernible and then find $\left(b_{i}'\right)$ in $A$
such that $b_{i}'=b_{i}$ if and only if $i$ is even, and $tp_{L}\left(\left(a_{i}b_{i}\right)_{i<\kappa}\right)=tp_{L}\left(\left(a_{i}b_{i}'\right)_{i<\kappa}\right)$,
so still $L$-indiscernible. But then $\left(M,A\right)\models f(a_{i},b_{i}',e)$
if and only if $i$ is even -- a contradiction to $M$ being NIP over
$A$ with respect to $L$-formulas.

So, by compactness we find some $R(x,y)\in L(A)$ and $k<\omega$
such that $\left(M,A\right)\models\forall x\in\P\, R(x,f(x))\land\exists^{\leq k}y\in\P\, R(x,y)$. 
\end{proof}

\begin{claim*}
In the previous claim, we can take $k=1$.\end{claim*}
\begin{proof}
Pillay's proof of \cite[Lemma 3.3]{MR2830421} goes through again,
with $\acl$, $\dcl$ and forking all considered inside of the $L$-induced
structure on $A$ (which is saturated and eliminates quantifiers).
\end{proof}
Finally, we conclude by induction on the dimension of the externally
definable sets. So let $X=A^{n+1}\cap\phi(x_{0},...,x_{n},x_{n+1},c)$
be given, and assume inductively that $A$ is uniformly $\left(n,\infty\right)$-stably
embedded (the base case given by the assumption). For any $a\in A$,
let $X_{a}=A^{n}\cap\phi(x_{0},...,x_{n},a,c)$. By the inductive
assumption, there is some $\psi(x_{0},...,x_{n},z)$ such that for
any $a\in A$, $X_{a}=A^{n}\cap\psi(x_{0},...,x_{n},b)$ for some
$b\in A$. By Shelah's expansion theorem, the function $f:\, A\to Z$
sending $a$ to $\left[b\right]_{\psi}$ (the canonical parameter
of $\psi(x_{0},...,x_{n},b)$) is externally definable. Thus, by the
previous claim, it is actually definable with parameters from $A$.
It follows that $X$ is defined by $\psi(x_{0},...,x_{n},f(x_{n+1}))$.
\end{proof}
As an application, we obtain a new proof of a theorem of Marker and
Steinhorn \cite{MR1264974}.
\begin{cor}
Let $T$ be $o$-minimal and $M\models T$. Assume that the order
on $M$ is complete. Then all types over $M$ are uniformly definable.
\end{cor}
\bibliographystyle{alpha}
\bibliography{Everything}

\end{document}